\documentclass[a4paper]{amsart}

\usepackage{amssymb}

\newtheorem{theorem}[equation]{Theorem}
\newtheorem*{maintheorem}{Main Theorem}

\numberwithin{equation}{section}

\theoremstyle{definition}

\newtheorem*{example*}{Example}

\newtheorem{remark}[equation]{Remark}
\newtheorem*{remark*}{Remark}

\newcommand{\bF}{{\mathbb F}}
\newcommand{\bZ}{{\mathbb Z}}

\newcommand{\bP}{{\mathbb P}}

\newcommand{\bO}{{\mathbb O}}

\newcommand{\frg}{{\mathfrak g}}

\newcommand{\frf}{{\mathfrak f}}

\newcommand{\calO}{{\mathcal O}}

\newcommand{\subo}{_{\bar 0}}
\newcommand{\subuno}{_{\bar 1}}

\DeclareMathOperator{\eespan}{span}
\providecommand{\espan}[1]{\eespan\left\{ #1\right\}}

 \newcommand{\tri}{\mathfrak{tri}}

 \newcommand{\frsl}{{\mathfrak{sl}}}
 
 \newcommand{\frso}{{\mathfrak{so}}}
 
 \newcommand{\frgl}{{\mathfrak{gl}}}

 \DeclareMathOperator{\ad}{ad}
 \DeclareMathOperator{\Ad}{Ad}
 \DeclareMathOperator{\der}{\mathfrak{der}}

 \DeclareMathOperator{\Aut}{Aut}
  \DeclareMathOperator{\Int}{Int}
 \DeclareMathOperator{\Fix}{Fix}

\def\bigstrut{\vrule height 14pt width 0ptdepth 2pt}

\newenvironment{romanenumerate}
 {\begin{enumerate}
 
 }{\end{enumerate}}

\begin{document}

\title{Jordan gradings on exceptional simple Lie algebras}

\author[Alberto Elduque]{Alberto Elduque$^{\star}$}
 \thanks{$^{\star}$ Supported by the Spanish Ministerio de
 Educaci\'{o}n y Ciencia
 and FEDER (MTM 2007-67884-C04-02) and by the
Diputaci\'on General de Arag\'on (Grupo de Investigaci\'on de
\'Algebra)}
 \address{Departamento de Matem\'aticas e
 Instituto Universitario de Matem\'aticas y Aplicaciones,
 Universidad de Zaragoza, 50009 Zaragoza, Spain}
 \email{elduque@unizar.es}


\date{October 15, 2008}

\subjclass[2000]{Primary 17B25}

\keywords{Jordan grading, simple, exceptional, Lie algebra, orthogonal decomposition}

\begin{abstract}
Models of all the gradings on the exceptional simple Lie algebras induced by Jordan subgroups of their groups of automorphisms are provided.
\end{abstract}

\maketitle


\section{Introduction}\label{se:Introduction}

In a recent paper \cite{Eld.GrSym}, some natural gradings on either  octonion algebras or Okubo algebras over fields of characteristic $\ne 2,3$ have been used to construct some nice gradings on the exceptional simple Lie algebras.
Okubo algebras constitute a class of eight dimensional non unital composition algebras. They are then endowed with a nondegenerate quadratic multiplicative form $n$ (so that $n(x*y)=n(x)n(y)$ for any $x,y$), and this form satisfies that the associated polar form $n(x,y)=n(x+y)-n(x)-n(y)$ is associative: $n(x*y,z)=n(x,y*z)$ for any $x,y,z$. These algebras were introduced by S.~Okubo \cite{Oku78} and have some remarkable features (see for instance \cite[Chapter 8]{KMRT}).

More precisely, the following gradings on exceptional simple Lie algebras were obtained in \cite{Eld.GrSym}:

\begin{enumerate}
\item A $\bZ_2^3$-grading on any octonion algebra $\bO$ induces a $\bZ_2^3$-grading on the simple Lie algebras $\frg$ of derivations of  $\bO$ (of type $G_2$), and also a $\bZ_2^3$-grading on the orthogonal simple Lie algebra $\hat\frg$ of the skew symmetric maps relative to the norm of $\bO$ (of type $D_4$), with $\frg_0=0=\hat\frg_0$ and such that $\frg_\alpha$ (respectively $\hat\frg_\alpha$) is a Cartan subalgebra of $\frg$ (resp. $\hat\frg$) for any $0\ne \alpha\in\bZ_2^3$.  (See \cite[Subsection 5.2]{Eld.GrSym}, and note that in many respects $D_4$ is exceptional.)

\item A $\bZ_3^3$-grading on any Okubo algebra $\calO$ induces a $\bZ_3^3$-grading on some attached simple Lie algebras $\frg$ and $\hat\frg$ of types $F_4$ and $E_6$,  with $\frg_0=0=\hat\frg_0$ and such that $\frg_\alpha$ (respectively $\hat\frg_\alpha$ is a two dimensional  subalgebra of $\frg$ (respectively three dimensional subalgebra of $\hat\frg$)  with $\frg_\alpha\oplus\frg_{-\alpha}$ (respectively $\hat\frg_\alpha\oplus \hat\frg_{-\alpha}$) being a Cartan subalgebra of $\frg$ (resp. $\hat\frg$) for any $0\ne \alpha\in\bZ_3^3$. (See
    \cite[Subsection 5.3]{Eld.GrSym}.)

\item A $\bZ_2^3$-grading on each of two octonion algebras induces a $\bZ_2^5$-grading on some attached simple Lie algebra $\frg$ of type $E_8$, with $\frg_0=0$ and such that $\frg_\alpha$ is a Cartan subalgebra of $\frg$ for any $0\ne \alpha\in\bZ_2^5$.  (See \cite[Subsection 5.4]{Eld.GrSym}.)
\end{enumerate}

\medskip

Recall \cite{Alek} that  given a simple Lie algebra $\frg$ and a complex Lie group $G$ with $\Int(\frg)\leq G\leq \Aut(\frg)$ (here $\Int(\frg)$ denotes the group of inner automorphisms and $\Aut(\frg)$ the group of all the automorphisms of $\frg$), an abelian subgroup $A$ of $G$ is a \emph{Jordan subgroup} if:
\begin{romanenumerate}
\item its normalizer $N_G(A)$ is finite,
\item $A$ is a minimal normal subgroup of its normalizer, and
\item its normalizer is maximal among the normalizers of those abelian subgroups satisfying (i) and (ii).
\end{romanenumerate}

The Jordan subgroups are shown in \cite{Alek} to be elementary (that is, isomorphic to $\bZ_p\times\cdots\times\bZ_p$ for some prime number $p$), and hence they induce gradings, called \emph{Jordan gradings}, in the Lie algebra $\frg$.

The classification of Jordan subgroups is given in \cite{Alek} in two tables. Table 1 deals with the classical Lie algebras. Detailed models of the corresponding Jordan gradings are given in \cite[Chapter 3, \S 3.12]{OV}.
On the other hand, Table 2 in \cite{Alek} gives the classification of the Jordan subgroups for the exceptional Lie algebras (see also \cite[Chapter 3, \S 3.13]{OV}). Table \ref{ta:JordanGradings} below summarizes some properties of these Jordan subgroups and of the corresponding Jordan grading. In all of them, the zero homogeneous subspace is trivial.

\begin{table}[h]
\begin{tabular}{|c|c|c|}
\hline
\bigstrut\qquad$\frg$\qquad\null&\qquad $A$\qquad\null &$\dim \frg_\alpha$ ($\alpha\ne 0$)\\[4pt]\hline
\bigstrut$G_2$&$\bZ_2^3$&$2$\\[2pt] \hline
\bigstrut$F_4$&$\bZ_3^3$&$2$\\[2pt] \hline
\bigstrut$E_8$&$\bZ_5^3$&$2$\\[2pt] \hline
\bigstrut$D_4$&$\bZ_2^3$&$4$\\[2pt] \hline
\bigstrut$E_8$&$\bZ_2^5$&$8$\\[2pt] \hline
\bigstrut$E_6$&$\bZ_3^3$&$3$\\[2pt] \hline
\end{tabular}
\caption{{\vrule height 16pt width 0pt}The exceptional Jordan gradings}
\label{ta:JordanGradings}
\end{table}

\smallskip

Therefore, the Jordan gradings on the exceptional simple complex Lie algebras, with the exception of a $\bZ_5^3$-grading on $E_8$, look like the gradings obtained from gradings on an octonion algebra or an Okubo algebra (see \cite[Remark 5.30]{Eld.GrSym}).

The first purpose of this paper is to check that those gradings induced by octonion and Okubo algebras are indeed Jordan gradings.

However, checking that a given subgroup of the automorphism group of a simple Lie algebra is a Jordan subgroup is not an easy task, so a different approach will be followed, which consists in proving the next result, which has its own independent interest:

\begin{maintheorem}
Let $\bF$ be an algebraically closed ground field of characteristic $0$. Then, up to equivalence:
\begin{romanenumerate}
\item There is a unique $\bZ_2^3$-grading on the simple Lie algebra of type $G_2$ over $\bF$ such that $\dim \frg_\alpha=2$ for any $0\ne\alpha\in \bZ_2^3$.

\item There is a unique $\bZ_3^3$-grading on the simple Lie algebra of type $F_4$  over $\bF$ such that $\dim \frg_\alpha=2$ for any $0\ne\alpha\in \bZ_3^3$.

\item There is a unique $\bZ_5^3$-grading on the simple Lie algebra of type $E_8$  over $\bF$ such that $\dim \frg_\alpha=2$ for any $0\ne\alpha\in \bZ_5^3$.

\item There is a unique $\bZ_2^3$-grading on the simple Lie algebra of type $D_4$  over $\bF$ such that $\dim \frg_\alpha=4$ for any $0\ne\alpha\in \bZ_2^3$.

\item There is a unique $\bZ_2^5$-grading on the simple Lie algebra of type $E_8$  over $\bF$ such that $\dim \frg_\alpha=8$ for any $0\ne\alpha\in \bZ_2^5$.

\item There is a unique $\bZ_3^3$-grading on the simple Lie algebra of type $E_6$  over $\bF$ such that $\dim \frg_\alpha=3$ for any $0\ne\alpha\in \bZ_3^3$.
\end{romanenumerate}
\end{maintheorem}

An immediate corollary of the Main Theorem is that indeed the gradings obtained from gradings on octonion or Okubo algebras are Jordan gradings.

\smallskip

Recall that two gradings $\frg=\oplus_{g\in G}\frg_g$ and $\frg=\oplus_{\gamma\in \Gamma}\frg_\gamma$ are said to be \emph{equivalent} if there is an automorphism $\varphi$ of $\frg$ such that for any $g\in G$ with $\frg_g\ne 0$, there is a $\gamma\in \Gamma$ with $\varphi(\frg_g)=\frg_{\gamma}$.

\smallskip

Actually, parts (i), (iv) and (v) of the Main Theorem follow from results by Hesselink. In fact, by \cite[Proposition 3.6]{Hesselink} any grading of a simple Lie algebra $\frg$ over $\bF$, with the properties of the gradings in the Main Theorem, satisfies that for any $0\ne\alpha$ in the grading group (which is $\bZ_p^r$ for $p=2$, $3$ or $5$, and $r=3$ or $5$) one has that the subspace
\begin{equation}\label{eq:g[alpha]}
\frg_{[\alpha]}=\oplus_{i=0}^{p-1}\frg_{i\alpha}
\end{equation}
is always a Cartan subalgebra of $\frg$. Now, by \cite[Theorem 6.2]{Hesselink}, these gradings are unique (up to equivalence) in cases (i), (iv) and (v) of the Main Theorem, where $\frg_\alpha$ is a Cartan subalgebra for any $\alpha\ne 0$. For $E_8$ this has been proved earlier in \cite{Thompson}, where it is shown that there is a unique (up to conjugation by automorphisms) \emph{Dempwolff decomposition} of $E_8$.

Moreover, as mentioned above, the gradings in parts (i), (iv) and (v) are all obtained from the natural $\bZ_2^3$-grading on the algebra of octonions $\bO$ (see \cite{Eld.GrOct} and \cite{Eld.GrSym}). The $\bZ_2^3$-grading on $G_2=\der\bO$ and on $D_4=\frso(\bO)$ are just the gradings induced from the one in $\bO$, while the $\bZ_2^5$-grading on $E_8$ is obtained from the model of $E_8$ as a direct sum of two copies of the triality Lie algebra of the octonions (which is isomorphic to $\frso(\bO)$) and three copies of the tensor product of two copies of the octonions: $E_8=\bigl(\tri(\bO)\oplus\tri(\bO)\bigr)
\oplus\iota_0(\bO\otimes\bO)\oplus\iota_1(\bO\otimes\bO)
\oplus\iota_2(\bO\otimes\bO)$ (see \cite{Eld.GrSym} and the references there in). This model of $E_8$ is naturally $\bZ_2^2$-graded, with the zero homogeneous part given by the direct sum of the two copies of the triality Lie algebra, and the nonzero homogeneous parts given by the three copies of the tensor product of two copies of $\bO$. And this $\bZ_2^2$-grading is now refined by means of the $\bZ_2^3$-grading of $\bO$ to get a $\bZ_2^5$-grading of $E_8$ with the required properties (see \cite[\S 5.4]{Eld.GrSym} for the details).

\smallskip

Therefore, the rest of the paper will be devoted to prove parts (ii), (iii), and (vi) of the Main Theorem. In the process, very concrete models of the corresponding Jordan gradings will emerge.

As a consequence, detailed models of all the Jordan gradings in Table 2 of \cite{Alek} (the exceptional Jordan gradings) are obtained.

\smallskip

To finish this introduction, note that all these gradings are related to the so called \emph{orthogonal decompositions}, introduced in \cite{KKU} (see \cite{KosTiep} and the references there in). For any of these gradings, if we denote by $\bP(\bZ_p^r)$ the projective space of dimension $r-1$ over the finite field $\bZ_p$, and if for $0\ne\alpha\in \bZ_p^r$, $[\alpha]$ denotes the corresponding point in $\bP(\bZ_p^r)$, then the subalgebras $\frg_{[\alpha]}$ in \eqref{eq:g[alpha]} are Cartan subalgebras of $\frg$, and the decomposition
\begin{equation}\label{eq:orthogonaldecomposition}
\frg=\bigoplus_{[\alpha]\in\bP(\bZ_p^r)}\frg_{[\alpha]}
\end{equation}
is a decomposition of $\frg$ into a direct sum of Cartan subalgebras which are orthogonal relative to the Killing form (as the homogeneous subspaces $\frg_\alpha$ and $\frg_\beta$ are always orthogonal unless $\beta=-\alpha$). That is, the decomposition in \eqref{eq:orthogonaldecomposition} is an orthogonal decomposition of $\frg$.

\smallskip

The next section will be devoted to prove the Main Theorem for the $\bZ_5^3$-gradings on $E_8$ (part (iii)), and then Section \ref{se:E6F4} will deal with parts (vi) and (ii).
\bigskip

\section{$\bZ_5^3$-grading on $E_8$}\label{se:E8}

The purpose of this section is to prove part (iii) of the Main Theorem, that is:

\begin{theorem} Let $\bF$ be an algebraically closed field of characteristic $0$ and let $\frg$ be the simple Lie algebra of type $E_8$ over $\bF$. Then up to equivalence there is a unique $\bZ_5^3$-grading of $\frg$ such that $\dim\frg_\alpha=2$ for any $0\ne\alpha\in \bZ_5^3$.
\end{theorem}
\begin{proof}
First note that by dimension count $\frg_0=0$ holds. The proof will follow several steps.

\bigskip

\noindent\textbf{Step 1:}\quad The construction of a suitable model of the simple Lie algebra of type $E_8$.

Let $V_1$ and $V_2$ be two vector spaces over $\bF$ of dimension $5$, and consider the $\bZ_5$-graded vector space
\begin{equation}\label{eq:E85graded}
\frg=\oplus_{i=0}^4\frg_{\bar\imath},
\end{equation}
where
\begin{equation}\label{eq:E8components}
\begin{split}
\frg\subo&=\frsl(V_1)\oplus \frsl(V_2),\\
\frg\subuno&=V_1\otimes \textstyle{\bigwedge^2}V_2,\\
\frg_{\bar 2}&=\textstyle{\bigwedge^2}V_1\otimes \textstyle{\bigwedge^4}V_2,\\
\frg_{\bar 3}&=\textstyle{\bigwedge^3}V_1\otimes V_2,\\
\frg_{\bar 4}&=\textstyle{\bigwedge^4}V_1\otimes \textstyle{\bigwedge^3}V_2.
\end{split}
\end{equation}
(All the tensor products are considered over the ground field $\bF$.)
This is a $\bZ_5$-graded Lie algebra, with the natural action of the semisimple algebra $\frg\subo$ on each of the other homogeneous components, and the brackets between elements in different components are given by  suitable scalar multiples of the only $\frg\subo$-invariant possibilities. In this way, $\frg$ is the exceptional simple Lie algebra of type $E_8$. The details of the Lie multiplication have been computed in \cite{CristinaE8}. This decomposition has received some attention lately \cite{KostantE8}.

\bigskip

\noindent\textbf{Step 2:}\quad Up to conjugation in $\Aut\frg$, there is a unique order $5$ automorphism of the simple Lie algebra $\frg$ of type $E_8$ such that the dimension of the subalgebra of fixed elements is $48$.

Actually, as shown in \cite[\S 8.6]{Kac}, up to conjugation, the finite order automorphisms of $E_8$ are in one to one correspondence with subsets of nodes of the affine Dynkin diagram $E_8^{(1)}$
\begin{equation}\label{eq:E81}
\begin{gathered}
\begin{picture}(160,50)(0,-30)
 \multiput(10,0)(20,0){8}{\circle{3}}
 \put(90,0){\circle*{3}}
 \put(110,-20){\circle{3}}
 \multiput(12,0)(20,0){7}{\line(1,0){16}}
 \put(110,-18){\line(0,1){16}}
 \put(10,8){\makebox(0,0){$1$}}
 \put(30,8){\makebox(0,0){$2$}}
 \put(50,8){\makebox(0,0){$3$}}
 \put(70,8){\makebox(0,0){$4$}}
 \put(90,8){\makebox(0,0){$5$}}
 \put(110,8){\makebox(0,0){$6$}}
 \put(130,8){\makebox(0,0){$4$}}
 \put(150,8){\makebox(0,0){$2$}}
 \put(118,-20){\makebox(0,0){$3$}}
\end{picture}
\end{gathered}
\end{equation}
such that the sum of the integers that label the nodes in the subset is exactly $5$. Given such a subset of, say, $r$ nodes, the fixed subalgebra is the direct sum of the semisimple Lie algebra whose Dynkin diagram is the one obtained by removing in the diagram \eqref{eq:E81} the nodes in the subset, and a center of dimension $r-1$. Now it is easy to see that the only possibility is the automorphism $\sigma$ obtained when considering the subset that consists exactly of the node with label $5$. In this case, one gets a $\bZ_5$-grading of $\frg$ where $\frg\subo$ is a direct sum of two copies of the simple Lie algebra of type $A_4$. The uniqueness show us that, up to conjugation, $\sigma$ is the automorphism of $\frg$ such that its restriction to $\frg_{\bar\imath}$ (notation as in Step 1) is $\xi^i$ times the identity, where $\xi$ is a fixed primitive fifth root of unity.

\bigskip

\noindent\textbf{Step 3:}\quad Assume that $\frg=\oplus_{0\ne \alpha\in \bZ_5^3}\frg_\alpha$ is a $\bZ_5^3$-graded simple Lie algebra of type $E_8$ with $\dim\frg_\alpha=2$ for any $0\ne\alpha\in\bZ_5^3$. The homogeneous spaces are given by the common eigenspaces of three commuting order $5$ automorphisms $\sigma_1$, $\sigma_2$, and $\sigma_3$, of $\frg$ which generate a subgroup of $\Aut\frg$ isomorphic to $\bZ_5^3$.

\smallskip

\noindent\textbf{3.1 ($\sigma_1$):}\quad Step 2 shows us that, without loss of generality, we may assume that $\sigma_1$ is the automorphism such that $\sigma_1(x)=\xi^ix$ for any  $x\in \frg_{\bar\imath}$ (notation as in Step 1).

\smallskip

\noindent\textbf{3.2 ($\sigma_2$):}\quad Consider now the order $5$ automorphism $\sigma_2$. As it commutes with $\sigma_1$, the restriction $\sigma_2\vert_{\frg\subo}$ is an automorphism of $\frg\subo$. Its order is then either $1$ or $5$. Given a subset of automorphisms of $\frg$, let us denote by $\Fix(S)$ the subset of elements fixed by all the elements in $S$. Note that
\[
\Fix(\sigma_2\vert_{\frg\subo})=\Fix(\sigma_1,\sigma_2)=\oplus_{i=0}^4\frg_{i\alpha}
\]
for some $0\ne\alpha\in\bZ_5^3$, and this subspace has dimension $8$. We conclude that $\sigma_2\vert_{\frg\subo}$ has order $5$. Since $\frsl(V_1)$ and $\frsl(V_2)$ are the only ideals of $\frg\subo$ and $\sigma_2$ induces a permutation of these two ideals of order $1$ or $5$, it follows that both $\frsl(V_1)$ and $\frsl(V_2)$ are invariant under the action of $\sigma_2$, and since $\dim\Fix(\sigma_1,\sigma_2)$ is $8$, it turns out that the restriction of $\sigma_2$ to $\frsl(V_i)$ has order $5$ ($i=1,2$).

Recall \cite[Chapter IX]{Jacobson} that $\Int(\frsl(V_i)$ is the group generated by the set $\{\exp\ad_a: a\in\frsl(V_i),\ a\textrm{\ nilpotent}\}$, and that the quotient $\Aut(\frsl(V_i))/\Int(\frsl(V_i))$ is a cyclic group of order $2$. Since the order of the restriction $\sigma_2\vert_{\frsl(V_i)}$ is $5$, this restriction belongs to $\Int(\frsl(V_i))$, $i=1,2$.

Therefore, there are nilpotent endomorphisms $a_{ij}\in \frsl(V_i)$, $j=1,\ldots,m_i$, $i=1,2$, such that
\[
\sigma_2\vert_{\frsl(V_i)}=\exp\ad_{a_{i1}}\cdots\exp\ad_{a_{im_i}}.
\]
Hence, the restriction $\sigma_2\vert_{\frg\subo}$ extends to the automorphism $\hat\sigma_2$ of $\frg$ given by the formula:
\[
\hat\sigma_2=\exp\ad_{11}\cdots\exp\ad_{1m_1}\exp\ad_{21}\cdots\exp\ad_{2m_2}.
\]
Note that $\hat\sigma_2$ leaves invariant the subspaces $\frg_{\bar\imath}$, for $0\leq i\leq 4$. Thus the automorphism $\hat\sigma_2^{-1}\sigma_2$ leaves invariant all the subspaces $\frg_{\bar\imath}$ and its restriction to $\frg\subo$ is the identity. But each $\frg_{\bar 1}$  is an irreducible module for $\frg\subo$, so Schur's Lemma shows that there is a nonzero scalar $\lambda\in\bF$ such that
\[
\hat\sigma_2^{-1}\sigma_2\vert_{\frg\subuno}=\lambda 1,
\]
and, as $\frg\subuno$ generates $\frg$ as a Lie algebra, it follows that the restriction of $\hat\sigma_2^{-1}\sigma_2$ to $\frg_{\bar\imath}$ is $\lambda^i$ times the identity map, and that $\lambda^5=1$. Also note that given any endomorphism $a\in\frsl(V_i)$, $\exp\ad_a=\Ad_{\exp a}$ on $\frsl(V_i)$ ($\Ad_g(x)=gxg^{-1}$ for any $g\in GL(V_i)$ and $a\in \frsl(V_i)$), while $\ad_a$ acts on each $\bigwedge^jV_i$ in the natural way, so that $\exp\ad_a$ acts on $\bigwedge^jV_i$ as $\wedge^j\exp a$ (where $(\wedge^jf)(w_1\wedge\cdots\wedge w_j)=f(w_1)\wedge\cdots\wedge f(w_j)$).

Consider the elements $b_{ij}=\exp a_{ij}\in SL(V_i)$, and $b_i=b_{i1}\cdots b_{im_i}$. Then the restrictions of $\hat\sigma_2$ to $\frsl(V_i)$ ($i=1,2$) and $\frg\subuno=V_1\otimes \bigwedge^2V_2$ are, respectively, the automorphism $\Ad_{b_i}$ and the linear isomorphism $b_1\otimes\wedge^2 b_2$. If $b_1$ is changed to $\lambda b_1$, then we get the new automorphism $\tilde\sigma_2$ such that $\tilde\sigma_2\vert_{\frg\subo}=\hat\sigma_2\vert_{\frg\subo}=\sigma_2\vert_{\frg\subo}$ and $\tilde\sigma_2\vert_{\frg\subuno}=\lambda\hat\sigma_2\vert_{\frg\subuno}=\sigma_2\vert_{\frg\subuno}$.
It follows that $\tilde\sigma_2=\sigma_2$ (recall that $\frg\subuno$ generates $\frg$).

Summarizing the previous arguments, it has been proven that there are elements $b_i\in SL(V_i)$, $i=1,2$ such that
\begin{equation}\label{eq:b1b2}
\sigma_2\vert_{\frsl(V_i)}=\Ad_{b_i}\quad (i=1,2),\qquad \sigma_2\vert_{\frg\subuno}=b_1\otimes\wedge^2b_2.
\end{equation}
Moreover, the order of $\sigma_2$ is $5$, so $(\sigma_2\vert_{\frg\subo})^5=1$, which implies that $b_i^5=\lambda_i1_{V_i}$ for some $0\ne \lambda_i\in\bF$, $i=1,2$. But also $(\sigma_2\vert_{\frg\subuno})^5=1$, whence $\lambda_1\lambda_2^2=1$. Since $\bF$ is algebraically closed, we can take scalars $\mu_1,\mu_2\in \bF$ such that $\mu_1^5=\lambda_1^{-1}$, $\mu_2^5=\lambda_2^{-1}$ and $\mu_1\mu_2^2=1$. We may substitute $b_i$ by $\mu_ib_i$, $i=1,2$ in \eqref{eq:b1b2}, and hence we may assume that $b_i^5=1_{V_i}$, $i=1,2$.

Besides, for $i=1,2$, since $b_i^5=1$, $b_i$ is a diagonalizable endomorphism of $V_i$ whose eigenvalues are fifth roots of unity. Note that the subspace $\{x\in \frsl(V_i): b_ixb_i^{-1}=x\}$ has dimension at least $4$, because the endomorphisms which act diagonally on a basis of eigenvectors of $b_i$ commute with $b_i$. But if an eigenvalue of $b_i$ has multiplicity $\geq 2$, then the dimension above is strictly greater that $4$, and this contradicts the dimension of $\Fix(\sigma_1,\sigma_2)=\Fix(\sigma_2\vert_{\frg\subo})$ being exactly $8$. Therefore, all the eigenvalues of $b_i$ have multiplicity $1$, and therefore a basis $\{v_{i1},\ldots,v_{i5}\}$ of $V_i$ can be taken with $b(v_{ij})=\xi^jv_{ij}$ ($i=1,2$, $j=1,2,3,4,5$). That is, the matrix of $b_i$ in this basis is precisely
\[
\begin{pmatrix} 1&0&0&0&0\\ 0&\xi&0&0&0\\ 0&0&\xi^2&0&0\\
0&0&0&\xi^3&0\\ 0&0&0&0&\xi^4
\end{pmatrix}.
\]

\smallskip

\noindent\textbf{3.3 ($\sigma_3$):}\quad Finally, let us consider the automorphism $\sigma_3$. With the same arguments used in \textbf{3.2}, elements $c_i\in SL(V_i)$ ($i=1,2$) can be found with $c_i^5=1$, no repeated eigenvalues, and such that
\[
\sigma_3\vert_{\frg\frsl(V_i)}=\Ad_{c_i},\qquad
\sigma_3\vert_{\frg\subuno}=c_1\otimes\wedge^2c_2.
\]
As $\sigma_2$ and $\sigma_3$ commute, it follows in particular that $\Ad_{b_i}\Ad_{c_i}=\Ad_{c_i}\Ad_{b_i}$ in $\frsl(V_i)$, or $b_ic_i=\mu_ic_ib_i$ for some $0\ne \mu_i\in \bF$. Since $b_i^5=1$, we have $\mu_i^5=1$, $i=1,2$.

But if $\mu_i$ were equal to $1$, then $c_i$ would belong to $\{x\in\frgl(V_i): xb_i=b_ix\}=\espan{b_i^j: j=0,\ldots,4}$, so the subspace $\{x\in\frsl(V_i): \sigma_2(x)=\sigma_3(x)=x\}=\{x\in\frsl(V_i):xb_i=b_ix\}$ would have dimension $4$, while we have
\[
\{x\in\frsl(V_i):\sigma_2(x)=\sigma_3(x)=x\}\subseteq \Fix(\sigma_1,\sigma_2,\sigma_3)=\frg_0=0,
\]
a contradiction. Therefore, $\mu_i\ne 1$, $i=1,2$.

We may change $\sigma_3$ by $\sigma_3^j$ for $1\leq j\leq 4$, which implies changing $c_i$ by the corresponding power), and in this way we may assume that $\mu_1=\xi$, the fixed primitive fifth root of unity we have been using so far. (Note that the grading induced by $\sigma_1,\sigma_2,\sigma_3$ is induced too by $\sigma_1,\sigma_2$ and $\sigma_3^j$.)

Moreover, the commutativity of $\sigma_2$ and $\sigma_3$ on $\frg\subuno$ gives:
\[
b_1c_1\otimes\wedge^2(b_2c_2)=c_1b_1\otimes\wedge^2(c_2b_2)= \mu_1\mu_2^2b_1c_1\otimes\wedge^2(b_2c_2),
\]
so that $\mu_1\mu_2^2=1$ and thus we may assume that $\mu_1=\xi$ and $\mu_2=\xi^2$.

Since $b_1c_1=\xi c_1b_1$, $b_1c_1^j(v_{1j})=\xi^jc_1(v_{1j})$ ($j=1,\ldots,5$), and hence we may scale the basic vectors $v_{1j}$ so that $c(v_{1j})=v_{1(j+1)}$ for $j=1,2,3,4$. In other words, a basis can be taken in $V_1$ such that the coordinate matrices of $b_1$ and $c_1$ are:
\begin{equation}\label{eq:b1c1}
b_1\leftrightarrow
\begin{pmatrix} 1&0&0&0&0\\ 0&\xi&0&0&0\\ 0&0&\xi^2&0&0\\
0&0&0&\xi^3&0\\ 0&0&0&0&\xi^4
\end{pmatrix},\qquad
c_1\leftrightarrow
\begin{pmatrix} 0&0&0&0&1\\ 1&0&0&0&0\\ 0&1&0&0&0\\
0&0&1&0&0\\ 0&0&0&1&0
\end{pmatrix}.
\end{equation}
In the same vein, since $b_2c_2=\xi^2c_2b_2$, permuting and scaling the previous basis on $V_2$, a new basis can be taken in $V_2$ such that the coordinate matrices of $b_2$ and $c_2$ are:
\begin{equation}\label{eq:b2c2}
b_2\leftrightarrow
\begin{pmatrix} 1&0&0&0&0\\ 0&\xi^2&0&0&0\\ 0&0&\xi^4&0&0\\
0&0&0&\xi&0\\ 0&0&0&0&\xi^3
\end{pmatrix},\qquad
c_2\leftrightarrow
\begin{pmatrix} 0&0&0&0&1\\ 1&0&0&0&0\\ 0&1&0&0&0\\
0&0&1&0&0\\ 0&0&0&1&0
\end{pmatrix}.
\end{equation}

\bigskip
In conclusion, up to equivalence, the only $\bZ_5^3$-grading of $\frg$ such that $\dim\frg_\alpha=2$ for any $0\ne\alpha\in\bZ_5^3$ is given by the automorphisms $\sigma_1,\sigma_2,\sigma_3$ such that
\begin{equation}\label{eq:sigmas}
\begin{split}
&\sigma_1(x)=\xi^ix\quad\textrm{for any $x\in \frg_{\bar\imath}$ and $0\leq i\leq 4$,}\\
&\sigma_2\vert_{\frg\subuno}=b_1\otimes\wedge^2b_2,\\
&\sigma_3\vert_{\frg\subuno}=c_1\otimes\wedge^2c_2,
\end{split}
\end{equation}
where the $\frg_{\bar\imath}$'s are the homogeneous components in \eqref{eq:E8components}, and on  fixed bases of $V_1$ and $V_2$, $b_1$ and $c_1$ (respectively $b_2$ and $c_2$) are the endomorphisms of $V_1$ (respectively $V_2$) in \eqref{eq:b1c1} (respectively \eqref{eq:b2c2}).
\end{proof}

\medskip

\begin{remark}
The proof of the previous Theorem gives a precise model for the $\bZ_5^3$ Jordan grading of $E_8$.
\end{remark}

\bigskip

\section{$\bZ_3^3$-gradings on $E_6$ and $F_4$}\label{se:E6F4}

In this section parts (ii) and (vi) of the Main Theorem will be proved. Many arguments are quite similar to the ones used for $E_8$, so they will just be sketched.

We start with $E_6$:

\begin{theorem}\label{th:E6} Let $\bF$ be an algebraically closed field of characteristic $0$ and let $\frg$ be the simple Lie algebra of type $E_6$ over $\bF$. Then up to equivalence there is a unique $\bZ_3^3$-grading of $\frg$ such that $\dim\frg_\alpha=3$ for any $0\ne\alpha\in \bZ_3^3$.
\end{theorem}
\begin{proof}
The same steps as for $E_8$ will be followed.

\bigskip

\noindent\textbf{Step 1:}\quad The construction of a suitable model of the simple Lie algebra of type $E_6$.

Here the model appears in \cite[Chapter 13]{Adams} (see also \cite[\S 3]{DraperModelsF4}). Let $V_1$, $V_2$ and $V_3$ be three vector spaces of dimension $3$ over $\bF$ and consider the $\bZ_3$-graded Lie algebra
\begin{equation}\label{eq:E63graded}
\frg=\frg\subo\oplus\frg\subuno\oplus\frg_{\bar 2},
\end{equation}
where
\begin{equation}\label{eq:E6components}
\begin{split}
\frg\subo&=\frsl(V_1)\oplus \frsl(V_2)\oplus\frsl(V_3),\\
\frg\subuno&=V_1\otimes V_2\otimes V_3,\\
\frg_{\bar 2}&=V_1^*\otimes V_2^*\otimes V_3^*.
\end{split}
\end{equation}
This is a $\bZ_3$-graded Lie algebra, with the natural action of the semisimple algebra $\frg\subo$ on each of the other homogeneous components, and the brackets between elements in different components are given by suitable scalar multiples of the only $\frg\subo$-invariant possibilities. In this way, $\frg$ is the exceptional simple Lie algebra of type $E_6$.

\bigskip

\noindent\textbf{Step 2:}\quad Up to conjugation in $\Aut\frg$, there is a unique order $3$ automorphism of the simple Lie algebra $\frg$ of type $E_6$ such that the dimension of the subalgebra of fixed elements is $24$.

Actually, this automorphism $\sigma$ is the one that corresponds to the only node labeled by $3$ in of the affine Dynkin diagram $E_6^{(1)}$:
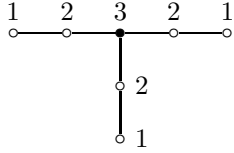
\[
\begin{picture}(100,70)(0,-50)
 \multiput(10,0)(20,0){5}{\circle{3}}
 \put(50,0){\circle*{3}}
 \multiput(50,-20)(0,-20){2}{\circle{3}}
 \multiput(12,0)(20,0){4}{\line(1,0){16}}
 \multiput(50,-38)(0,20){2}{\line(0,1){16}}
 \put(10,8){\makebox(0,0){$1$}}
 \put(30,8){\makebox(0,0){$2$}}
 \put(50,8){\makebox(0,0){$3$}}
 \put(70,8){\makebox(0,0){$2$}}
 \put(90,8){\makebox(0,0){$1$}}
 \put(58,-20){\makebox(0,0){$2$}}
 \put(58,-40){\makebox(0,0){$1$}}
\end{picture}
\]
The uniqueness show us that, up to conjugation, $\sigma$ is the automorphism of $\frg$ such that its restriction to $\frg_{\bar\imath}$ (notation as in Step 1) is $\omega^i$ times the identity, where $\omega$ is a fixed primitive third root of unity.

\bigskip

\noindent\textbf{Step 3:}\quad Assume that $\frg=\oplus_{0\ne \alpha\in \bZ_3^3}\frg_\alpha$ is a $\bZ_3^3$-graded simple Lie algebra of type $E_6$ with $\dim\frg_\alpha=3$ for any $0\ne\alpha\in\bZ_3^3$. The homogeneous spaces are given by the common eigenspaces of three commuting order $3$ automorphisms $\sigma_1$, $\sigma_2$, and $\sigma_3$, of $\frg$ which generate a subgroup of $\Aut\frg$ isomorphic to $\bZ_3^3$.

\smallskip

\noindent\textbf{3.1 ($\sigma_1$):}\quad Step 2 shows us that, without loss of generality, we may assume that $\sigma_1$ is the automorphism such that $\sigma_1(x)=\omega^ix$ for any  $x\in \frg_{\bar\imath}$ (notation as in Step 1).

\smallskip

\noindent\textbf{3.2 ($\sigma_2$):}\quad As for $E_8$, the restriction $\sigma_2\vert_{\frg\subo}$ is an order $3$ automorphism (otherwise the dimension of $\Fix(\sigma_1,\sigma_2)$ would be $>6$). Now, $\sigma_2$ induces a permutation of the three simple ideals of $\frg\subo$ of order $1$ or $3$, but if the order were $3$, then the eight dimensional subspace $\{x+\sigma_2(x)+\sigma_2^2(x): x\in\frsl(V_1)\}$ would be contained in the six dimensional subspace $\Fix(\sigma_1,\sigma_2)$, a contradiction. Therefore, $\sigma_2$ leaves invariant $\frsl(V_i)$ for all $i$.

Now the same arguments as for $E_8$ show that one may find elements $b_i\in SL(V_i)$, $i=1,2,3$, such that $b_i^3=1$ and
\[
\sigma_2\vert_{\frsl(V_i)}=\Ad_{b_i}\quad (i=1,2,3),\qquad \sigma_2\vert_{\frg\subuno}=b_1\otimes b_2\otimes b_3.
\]
Moreover, the minimal polynomial of $b_i\in SL(V_i)$ is exactly $X^3-1$ (its eigenvalues have multiplicity $1$).

\smallskip

\noindent\textbf{3.3 ($\sigma_3$):}\quad In the same vein, there are endomorphisms $c_i\in SL(V_i)$, $i=1,2,3$, with minimal polynomial $X^3-1$ and such that
\[
\sigma_3\vert_{\frg\frsl(V_i)}=\Ad_{c_i},\qquad
\sigma_3\vert_{\frg\subuno}=c_1\otimes c_2\otimes c_3.
\]
As for $E_8$, the commutation of $\sigma_2$ and $\sigma_3$ and a dimension count show that $b_ic_i=\mu_ic_ib_i$, with $1\ne \mu_i\in \bF$ and $\mu_i^3=1$, ($i=1,2,3$). Hence $\mu_i\in\{\omega,\omega^2\}$. Changing $\sigma_3$ by $\sigma_3^2$ if necessary, it can be assumed that $\mu_1=\omega$.

Moreover, the commutativity of $\sigma_2$ and $\sigma_3$ on $\frg\subuno$ forces the equality $\mu_1\mu_2\mu_3=1$, or $\mu_2\mu_3=\omega^2$. We conclude that $\mu_1=\mu_2=\mu_3=\omega$.
Hence, a basis can be chosen on each $V_i$ such that the coordinate matrices of $b_i$ and $c_i$ are:
\begin{equation}\label{eq:biciE6}
b_i\leftrightarrow
\begin{pmatrix} 1&0&0\\ 0&\omega&0\\ 0&0&\omega^2
\end{pmatrix},\qquad
c_i\leftrightarrow
\begin{pmatrix} 0&0&1\\ 1&0&0\\ 0&1&0
\end{pmatrix}.
\end{equation}

\bigskip
In conclusion, up to equivalence, the only $\bZ_3^3$-grading of $\frg$ such that $\dim\frg_\alpha=3$ for any $0\ne\alpha\in\bZ_3^3$ is given by the automorphisms $\sigma_1,\sigma_2,\sigma_3$ such that
\[
\begin{split}
&\sigma_1(x)=\omega^ix\quad\textrm{for any $x\in \frg_{\bar\imath}$ and $i=0,1,2$,}\\
&\sigma_2\vert_{\frg\subuno}=b_1\otimes b_2\otimes b_3,\\
&\sigma_3\vert_{\frg\subuno}=c_1\otimes c_2\otimes c_3,
\end{split}
\]
where the $\frg_{\bar\imath}$'s are the homogeneous components in \eqref{eq:E6components}, and on  fixed bases of $V_1$, $V_2$  and $V_3$, $b_i$ and $c_i$  are the endomorphisms of $V_i$ in \eqref{eq:biciE6} .
\end{proof}

\bigskip

The corresponding result for $F_4$ is the following:

\begin{theorem}\label{th:F4} Let $\bF$ be an algebraically closed field of characteristic $0$ and let $\frg$ be the simple Lie algebra of type $F_4$ over $\bF$. Then up to equivalence there is a unique $\bZ_3^3$-grading of $\frg$ such that $\dim\frg_\alpha=2$ for any $0\ne\alpha\in \bZ_3^3$.
\end{theorem}
\begin{proof} Here we will be even more sketchy, since the situation is simpler.

Just consider the model of $E_6$ obtained above, and consider the order $2$ automorphism $\tau$ which permutes $V_2$ and $V_3$. The subalgebra $\frg$ of elements fixed by $\tau$ is a simple Lie algebra of type $F_4$ (see \cite[\S 3]{DraperModelsF4}). Therefore, $\frg$ appears as the $\bZ_3$-graded Lie algebra
\begin{equation}\label{eq:F43graded}
\frg=\frg\subo\oplus\frg\subuno\oplus\frg_{\bar 2},
\end{equation}
\begin{equation}\label{eq:F4components}
\begin{split}
\frg\subo&=\frsl(V_1)\oplus \frsl(V_2),\\
\frg\subuno&=V_1\otimes S^2(V_2),\\
\frg_{\bar 2}&=V_1^*\otimes S^2(V_2^*).
\end{split}
\end{equation}
Here $S^2(V)$ denotes the subspace of symmetric tensors in $V\otimes V$.

\bigskip

Up to conjugation in $\Aut\frg$, there is a unique order $3$ automorphism of the simple Lie algebra $\frg$ of type $F_4$ such that the dimension of the subalgebra of fixed elements is $16$. Actually, this automorphism $\sigma$ is the one that corresponds to the only node labeled by $3$ in of the affine Dynkin diagram $F_4^{(1)}$:
\[
\begin{picture}(100,35)(0,-10)
\multiput(10,0)(20,0){5}{\circle{3}}
 \put(50,0){\circle*{3}}
 \multiput(12,0)(20,0){2}{\line(1,0){16}}
 \put(72,0){\line(1,0){16}}
 \multiput(52,-1)(0,2){2}{\line(1,0){16}}
 \put(60,0){\makebox(0,0){$>$}}
 \put(10,8){\makebox(0,0){$1$}}
 \put(30,8){\makebox(0,0){$2$}}
 \put(50,8){\makebox(0,0){$3$}}
 \put(70,8){\makebox(0,0){$4$}}
 \put(90,8){\makebox(0,0){$2$}}
\end{picture}
\]
Now, the same type of arguments as for $E_6$ give the result.
\end{proof}

\bigskip

Again, the proofs of Theorems \ref{th:E6} and \ref{th:F4} give precise models of the corresponding Jordan gradings. In \cite[\S 7]{DraperMartinF4} it is shown that for $F_4$ this grading is fine. These are different from the models obtained in \cite[\S 5.3]{Eld.GrSym}, which are based on a $\bZ_3^2$-grading of the Okubo algebra over $\bF$, which is complemented by an extra order three automorphism induced by the triality automorphism associated to the Okubo algebra. For $E_6$, this unique $\bZ_3^3$-grading is not fine, as the construction of $E_6$ in terms of an Okubo algebra requires the use of another two dimensional symmetric composition algebra, which in turn can be graded over $\bZ_3$, and used to get a $\bZ_3^4$-grading on $E_6$ (see \cite{Eld.GrSym} for details).

\bigskip



\def\cprime{$'$}
\def\cfudot#1{{\oalign{\accent94 #1\crcr\hidewidth.\hidewidth}}}

\providecommand{\bysame}{\leavevmode\hbox to3em{\hrulefill}\thinspace}
\providecommand{\MR}{\relax\ifhmode\unskip\space\fi MR }

\end{document}